\documentclass[11pt]{amsart}
\usepackage{amsmath, amsrefs, amsthm, amssymb}
\usepackage{graphicx}
\usepackage{caption}
\usepackage{subcaption}
\usepackage{pdfsync}	
\usepackage{dsfont}
\usepackage{setspace}%%%%DOBLE ESPACIO
\usepackage{keyval} 
\usepackage{float}				%search for a pdf file
\newcommand\R{{\mathbb{R}}}

\newtheorem{theorem}{Theorem}[section]
\newtheorem{proposition}{Proposition}[section]
\newtheorem{lemma}[theorem]{Lemma}

\theoremstyle{definition}
\newtheorem{definition}[theorem]{Definition}

\theoremstyle{remark}
\newtheorem{remark}[theorem]{Remark}
\numberwithin{equation}{section}

\def\R{\mathbb{R}}

\begin{document}
\title[The cubic heat equation]
{slowly oscillating solution of the cubic heat equation}
\author{ Fernando Cortez}
\address{Universit\'e de Lyon, Universit\'e Lyon 1,
CNRS UMR 5208 Institut Camille Jordan,
43 bd. du 11 novembre,
Villeurbanne Cedex F-69622, France.}
\email{cortez@math.univ-lyon1.fr}
\urladdr{http://math.univ-lyon1.fr/$\sim$cortez}
\begin{abstract}
In this paper, we are considering the Cauchy problem of the nonlinear heat equation  $u_t -\Delta u= u^{3 },\ u(0,x)=u_0$. After extending Y. Meyer's result establishing the  existence of global  solutions, under a smallness condition of  the initial data in the homogeneous Besov spaces $\dot{B}_{p}^{-\sigma, \infty}(\mathbb{R}^{3})$, where $3<p<9$ and $\sigma=1-3/p$, we prove that  initial data $u_0\in \mathcal{S}(\mathbb{R}^{3})$, arbitrarily small in ${\dot B^{-2/3,\infty}_{9}}(\mathbb{R}^{3})$, can produce solutions that explode in finite time. In addition, the blowup may occur after an arbitrarily short time.
\end{abstract}
\maketitle
\section{Introduction}
 A well-studied evolution equation is $\partial_t u= \Delta\phi(u) +f(u)$, for various choices of $\phi$ and $f$ (see \cite{Ball, Frie, Frie1, Gala84, Gala, Henr, Weis}). 
 An extensive bibliography  exists for the particular  case, when $\phi(u)=u$ and $f(u)= \left | u \right |^{\alpha} u$, where $\alpha>0$. Then, we have  the following Cauchy problem:
\begin{equation}
\label{Ter}
\left \{ \begin{matrix} \partial_t u  = \Delta u  +  \left | u \right |^{\alpha} u &   \ x\in \mathbb{R}^{n}  & t\in[0,T] 
\\ 
u(0,x) = u_0(x), \end{matrix}\right. 
\end{equation}
where $0<T\leq\infty$, $\alpha>0$ and $u:\mathbb{R}^{+}\times \mathbb{R}^{n} \longrightarrow \mathbb{R}$ a real function.
\\
The Duhamel formulation of \eqref{Ter} reads
\begin{equation}
 \label{c3}
   u(t)= e^{t \Delta} u_0(x) +  \displaystyle\int_{0}^{t} e^{(t-\tau)\Delta}  \left | u \right |^{\alpha} u(\tau) \, d\tau, 
  \end{equation}
  where, $e^{t\Delta}$ ($t \geq 0$) denotes the heat semigroup.  We have $e^{t\Delta} f = G_t *f$, where 
\[
G_t(x)= \frac{1}{\sqrt{4\pi t}} e^{-\frac{\left | x \right |^{2}}{4t}}. 
\]
By standard results, Cauchy problem \eqref{Ter} is well-posed in many Banach spaces. In particular, thanks to the work of F. Weissler, H. Brezis and T. Cazenave \cite{Weis1,Weis,BreCaz}, we know the following statements.
\begin{itemize}
\item  When $p > \frac{n\alpha}{2}$, $p\geq\alpha+1$, there exists a constant $T=T(u_0)$ and a unique  solution $u(t)\in C([0,T],L^{p}(\mathbb{R}^{n}))$. Also $u(t) \in L^{\infty}_{loc}(]0,T[,L^{\infty})$. 

\item  When $p = \frac{n\alpha}{2}=p_0$, $p\geq\alpha+1$, there exists a constant $T=T(u_0)$ and a unique  solution $u(t)\in C([0,T],L^{p_0}(\mathbb{R}^{n})) \cap L^{\infty}_{loc}(]0,T[,L^{\infty})$.

\item  When $\alpha + 1 < p < \frac{n\alpha}{2}$, there is no general theory of existence. Besides, A. Haraux and F. Weissler \cite{Haraux} showed that there is a solution belonging to the space $C([0,T],L^{p}(\mathbb{R}^{n})) \cap L^{\infty}_{loc}(]0,T[,L^{\infty})$, positive, arising from the initial data $0$, thus there is no uniqueness.
\end{itemize}
We will be interested in the issues of the blowup in finite time and of the global existence of the solutions.
The first works related to these kinds of questions are due to Hiroshi Fujita in 1966. Fujita has shown that for the positive solutions of \eqref{Ter}, if the initial data $u_0$ is of class $C^{2}(\mathbb{R}^{n})$ with its derivatives of order $0,\ 1$ and $2$ bounded on $\mathbb{R}^{n}$, then we have the following necessary condition for that $u$ to be unique  in
$C^{0}(\mathbb{R}^{n}\times [0,T)):$
\begin{eqnarray*}
\exists  \ M>0,  \ \exists  \ 0<\beta<2: \forall  x\in \mathbb{R}^{n} \ \  |u_0(x)| \leq M \ e^{|x|^{\beta}}.
\end{eqnarray*}
This means that $u_0$ should not grow too fast (see \cite{Fuji,Fuji1}).
\\
In regards to  the question of the existence of regular global solutions under small initial data assumptions, Fujita concluded that there are two types of situations:
if  $\alpha < \frac{2}{n}$, then no nontrivial positive solution of this problem which can  be global (Fujita phenomenon), while for $\alpha > \frac{2}{n}$, there are global non-trivial solutions in positive small initial data assumptions.
Years later, K. Hayakawa \cite{Hay} and F. Weissler \cite{Weis1,Weis} completed the study Fujita demonstrating that the  Fujita critical exponent $\alpha = \frac{2}{n}$ verifies The Fujita phenomenon. 
On the other hand, in the case of a homogeneous Dirichlet condition in an exterior smooth domain 
$\Omega$, Bandle et Levine studied the classical positive solutions which satisfy the following condition on the order of growth:
\begin{eqnarray*}
\forall k>0, |u(x,t)| \ e^{-k |x|} \rightarrow 0 \quad \mbox{and}  \quad |\nabla u(x,t)| \  e^{-k |x|} \rightarrow 0 \quad \mbox{when}  \ \ |x| \rightarrow 0.
 \end{eqnarray*} 
Equivalently to the previous problem, Bandle and Levine showed similar results to those of Fujita with  Dirichlet boundary conditions for the problem \eqref{Ter} (see \cite{Ban,Ban1}).
It was not long before the Fujita critical case was resolved by Ryuichi Suzuki \cite{Suz}. He proved
that the Fujita critical exponent $\alpha = \frac{2}{n}$ verifies The Fujita phenomenon, when $n\geq 3$.
\\ H. Levine and Q. Zhang addressed the same problems in the case of  Neumann boundary conditions (see \cite{Levi}). They considered an initial condition $u_0$in $C^{2}(\Omega)$  and weak solutions in the sense of distributions with the test space $C^{2}(\Omega)$ which are not subject to any restriction on the growth order.
They  showed similar results to those of Fujita and  that  for the Fujita critical exponent $\alpha = \frac{2}{n}$, the solution verifies The Fujita phenomenon. Similar results with Robin boundary conditions has been shown by Rault in \cite{Rau1}.
\\ To motivate our results, we introduce the concept of a scale-invariant space. Let $\lambda>0$, then we define 
\begin{eqnarray}
\label{scaling}
u_\lambda(t,x)= \lambda^{\frac{2}{\alpha}} u(\lambda^{2}t,\lambda x)  \qquad \mbox{and} \qquad u_{0,{\lambda}}(x)= \lambda^{\frac{2}{\alpha}} u_0(\lambda x).
\end{eqnarray}
For every solution $u(t,x)$ of \eqref{Ter}, $u_\lambda(t,x)$ is also a solution of  \eqref{Ter} for which the initial condition is $u_{0,{\lambda}}(x)$. In this case we say that a  Banach space $E$ is \textit{scale-invariant space}, if
\begin{eqnarray}
\label{invsca}
\left \| u(t,\cdot) \right \|_{E}= \left \| u_\lambda(t,\cdot) \right \|_{E}. 
\end{eqnarray}  
The spaces which are invariant under such a scaling are called invariant spaces for this class of non-linear heat equation. It is known that the \textit{scale-invariant space} plays an essential role in questions like: well-posedness, global existence or blow-up of the solution. The critical Lebesgue is $L^{p_0}(\mathbb{R}^{n})$, with $p_0=\frac{n\alpha}{2}$. Notice that $p_0\geq 1$ if and only if  $\alpha$ is larger or equal to the Fujita critical exponent.
\\The purpose of the present paper is to study the borderline cases of explosion and global existence for solutions of a
 particular case of \eqref{Ter}, in a scale-invariant Banach space.
Actually, we consider the cubic heat  equation  
\begin{eqnarray}
\label{class}
\left \{ \begin{matrix} \partial_t u  = \Delta u  + u^{3} &   \ x\in \mathbb{R}^{3}  & t\in[0,T] 
\\ 
u(0,x) = u_0(x), \end{matrix}\right. 
%\label{1}
\end{eqnarray} 
where $0< T \leq \infty$, where $u=u(x,t)$ is a real value function of  $(x,t), x\in \mathbb{R}^{3}$ and $t\geq 0$. 
One rewrites Equation \eqref{class} in the equivalent   Duhamel formulation
\begin{eqnarray}
\label{duhamel}
u(t,x)=u(t) = e^{t \Delta} u_0(x) +  \displaystyle\int_{0}^{t} e^{(t-\tau)\Delta}  u^{3}(\tau,x) \, d\tau.	
%\label{2}
\end{eqnarray}
The following proposition shows the equivalence between differential \eqref{class} and integral formulation \eqref{duhamel}.
\begin{proposition}
Let $u\in L^{3}([0,T],L_{loc}^{3}(\mathbb{R}^{3}))$. Then the following statements are equivalent
\begin{enumerate}
\item $u$ satisfies $\partial_t u = \Delta u + u^{3}$ in the sense of distributions;
\item  there exists $u_0 \in \mathcal{S}'$ such that $u(t) = e^{t\Delta}u_0 + \int_{0}^{t} e^{(t-\tau)\Delta} u^{3}(\tau) \,d\tau$.
\end{enumerate}
\end{proposition}
\begin{remark}
The last proposition shows that if the initial data $u_0\in L^{3}_{loc}(\mathbb{R}^{3})$, then a classical solution of \eqref{class} is equivalent in the sense of distributions to a mild solution of \eqref{class} with  the initial data $u_0$. In the following, we always talk about the existence of  a mild solution of  \eqref{class}.
\end{remark}
In the case of equation \eqref{class} the only Lebesgue space invariant under this scaling \eqref{scaling} is $L^{3}(\mathbb{R}^{3})$. Other examples of invariant spaces  for \eqref{class} that will play an important  role later on are:
\begin{eqnarray}
\label{inega}
\dot{H}^{\frac{1}{2}}  \hookrightarrow L^{3} \hookrightarrow \dot{B}^{-1+3/p,\infty}_{p}  \hookrightarrow \dot{B}^{-1,\infty}_{\infty}, \quad \mbox{for} \quad 3 < p< \infty.
\end{eqnarray}
Problem \eqref{class} shares some similarities  with 
the incompressible Navier-Stokes equation.
Recall that the Cauchy problem of the incompressible Navier-Stokes equation in $\mathbb{R}^{3} \times \mathbb{R}^{+}$ is
\begin{equation}
\label{N-S}
  \begin{cases}
   \partial_t u + u \cdot \nabla u- \Delta u  = - \nabla p \quad \quad   x\in\mathbb{R}^{3} \quad t>0,\\
   \mbox{div} \ u=0, \quad \quad \quad \quad \quad \quad \quad \quad \quad  x\in\mathbb{R}^{3},\\
     u(x,0)=u_0(x) \quad \quad \quad \quad \quad   \quad \quad  x\in\mathbb{R}^{3},  
  \end{cases}
  \end{equation}
where $u = u(t, x)$ is a vector with $3$ components representing the velocity of an incompressible fluid 
and $p (t, x)$ is a function representing the  pressure.
Similarly to the equation \eqref{class}, we can rewrite the system \eqref{N-S} in the following integral form
\[
u(t,x)=u(t) = e^{t \Delta} u_0 +  \displaystyle\int_{0}^{t} e^{(t-\tau)\Delta} \mathbb{P} \ \mbox{div} (u \otimes u) (s) \, ds,	
\tag{NS}
\]
where, $\mbox{div} \ u_0 =0$, $e^{t \Delta}$ is the heat semigroup, and $\mathbb{P}$ is the Leray-Hopf projection operator into divergence free vector, defined by
\begin{eqnarray*}
\mathbb{P} f =  f - \nabla \Delta^{-1} (\mbox{div} \ f).
\end{eqnarray*}
We can see that in the integral formula for the incompressible Navier-Stokes equation the term pressure $p(x,t)$ is gone. Indeed, the pressure can be recalculated from the velocity field $u(x,t)$ (see \cite{Way}).
\\ The equation (NS) has exactly the same scaling law that equation~\eqref{duhamel}. For both equations, it is possible  to establish the global existence of solutions in certain  homogeneous Besov spaces with small initial conditions(see \cite{Meyer}), and  the uniqueness of them in a suitable subspace of $C_t(L^3(\R^3))$,  where these solutions are built.  
A number of methods developed to this equation (as those described in the book \cite{Canno1}) can be transposed to the case of equation \eqref{duhamel}.
Of course, in other cases, the results for the Navier-Stokes equations do not easily fit into the equation of the cubic heat \eqref{class}.
\\  
This paper is organized as follows. In the next section we start by introducing the relevant notations and function spaces,  recalling a few basic results.  In the Section 3 will give the statements of our results. First an extension of a theorem by  Meyer on the existence of global solutions under small initial data assumptions. Next we state the main theorem \eqref{inflation}. In section 4 we give the proof of the Meyer theorem  and some comments. In the last section we prove main theorem and some comments.
\section{Preliminaries}
%We are interested in the questions about the sufficient conditions of explosion or global existence of the solution in small initial data assumptions. In the previous section, we discussed this question, but with sufficiently regular initial data and related solutions are therefore strong solutions (standard solutions), i.e, the functions which solve \eqref{class} for all $t \in[0, T]$ in the classic sense (with the classical notion of derivative).
First, we set the precise mathematical framework for the study of the Cauchy problem for the equation of the nonlinear heat \eqref{class}.
%\subsection{Theoretical framework and notations}
I%n the following, we will say that given a Banach space $\mathcal{Z}$, the Cauchy problem \eqref{class}  is well posed in $\mathcal{Z}$, if there is a Banach space $Y\subset C([0,T),\mathcal{Z})$ such that for all $u_0\in X$, there exists a unique solution $u\in Y$ of \eqref{class}  and the function $u_0 \longmapsto u=\Phi(u_0)$ is continues to $\mathcal{Z}$ in $C([0, T), \mathcal{Z})$. Otherwise, we say that the Cauchy problem is ill-posed in $\mathcal{Z}$. 

%In this part we join the Weissler, Brezis and Cazenave's theorems \cite{Weis1,Weis,BreCaz} in one theorems which allow us clearly establish the local existence theorem and uniqueness of the solution \eqref{class} and its properties. We do this because our calculations are based on the following  well-posedness theorem.
In the following theorem we called some of the results of  \cite{BreCaz,Weis1,Weis}
\begin{theorem}{\textit{[Weissler-Brezis-Cazenave \cite{BreCaz,Weis1,Weis}]}}
\label{localexis}
Let $\ u_0\in L^{3}(\mathbb{R}^{3})$. There exists a constant $T=T(u_0)$ and a unique of \eqref{class} in  $C([0,T],L^{3}(\mathbb{R}^{3})) \cap L^{\infty}_{loc}((0,T], L^{\infty}(\mathbb{R}^{3}))$, such that
\begin{itemize}
\item [(i)] $u(t,x)$ is classical solution of the Cauchy problem \eqref{class} on (0,T], \\
\item [(ii)] $\underset{0<t<T}{\sup} t^{\frac{\sigma}{2}}  \left \| u(\cdot,t) \right \|_{L^{p}}  < +\infty$, \\
\item [(iii)]  $\underset{t\rightarrow 0}{\lim} \ � t^{\frac{\sigma}{2}}  \left \| u(\cdot,t) \right \|_{L^{p}}=0$, \\
\end{itemize}
where $3<p<9$ and $\sigma=1-\frac{3}{p}$.
\end{theorem}
\begin{remark}
\label{L1}
If in addition, we consider $u_0\in L^{1}(\mathbb{R}^3) \cap L^{3}(\mathbb{R}^3)$, the unique  Weissler's solution $u(t,x)$ arising from $u_0$ verifies:  $u(t,\cdot)\in L^{1}(\mathbb{R}^3) \  \forall t \in [0,T]$. This observation readily follows from Weissler method and will be implicitly in some of our calculations, for example when we use  the Fourier transform of the solution. 
\end{remark}
Before stating our results, we define  the inhomogeneous and homogeneous  Besov spaces which  play an important role in our estimates. Recall the Littlewood-Paley decomposition. Let $\psi, \varphi \in \mathcal{S}(\mathbb{R}^{3})$ such that:
\begin{eqnarray*}
&& \mbox{supp} \  \widehat{\varphi} \subset \{ \left | \xi \right | \leq 5/6\} \qquad \mbox{and} \qquad \mbox{supp} \  \widehat{\psi} \subset \{ 3/5 \leq \left | \xi \right | \leq 5/3\}  \ \ \ \psi_j=2^{nj} \psi(2^{j}x), \ j\in \mathbb{Z}  \\
&&  1=\widehat{\varphi}(\xi) +  \displaystyle\sum_{j=0}^{\infty}  \widehat{\psi}_j (\xi) \ \ \ (\xi \in \mathbb{R}^{n})  \\
&&  1=\displaystyle\sum_{j=-\infty}^{\infty}  \widehat{\psi}_j (\xi) \ \ \ (\xi \in \mathbb{R}^{n})
\end{eqnarray*}
where $\widehat{f}$ denotes the Fourier transform of $f$.
\begin{definition}
The inhomogeneous and homogeneous Besov spaces $B^{s,q}_{p}$ and $\dot{B}^{s.q}_{p}$ are defined as follows (at least for $s<0$, which will be our case):
\begin{eqnarray*}
 \dot{B}^{s.q}_{p} &=& \{ f\in\mathcal{S}'(\mathbb{R}^{3});  \left \| f \right \|_{\dot{B}^{s,q}_{p}} <  \infty  \}, \ \ \ \ \{ f\in\mathcal{S}'(\mathbb{R}^{3});  \left \| f \right \|_{\dot{B}^{s,q}_{p}} <  \infty  \}, \\
\left \| f \right \|_{B^{s,q}_{p}} &=& \left \| \varphi * f\ \right \|_{p} +\left(\displaystyle\sum_{j=0}^{\infty}  \left \| 2^{js} \psi_j * f  \right \|_{p}^{q}\right)^{1/q} \\
\left \| f \right \|_{\dot{B}^{s,q}_{p}} &=& \left(\displaystyle\sum_{j=-\infty}^{\infty}  \left \| 2^{js} \psi_j * f  \right \|_{p}^{q}\right)^{1/q}
\end{eqnarray*}
for $s\leq0 \ \ 1\leq p, q \leq \infty$.
\end{definition}
\section{Quick overview of the main results}
In this section, we will give the main result of the blow-up of the solution of \eqref{class}. But first, we will give an extension of Meyer theorem to the case of global solutions of the nonlinear heat equation with a smallness assumption on the initial condition.
\\
The general methodology that will be used throughout this section is to look for the solutions $u(t,x)$ belonging to the Banach space $\mathcal{X} = C_b([0,\infty); \mathcal{Z})$, where $\mathcal{Z}$ is a suitable functional Banach spaces.  
\\The norm of $u(\cdot,t)$ in $\mathcal{X} = C_b([0,\infty); \mathcal{Z})$ is denoted by $\left\| u \right \|_{\mathcal{X}}$ and defined as
\begin{eqnarray}
 \left \| u \right \|_{\mathcal{X}} = \underset{t>0}{\sup} \  \left \| u(\cdot,t) \right \|_{\mathcal{Z}}.
\end{eqnarray}
This norm will be called the natural norm. To prove our Theorem \eqref{T1}, we first assume the existence and uniqueness of a local solution from initial data $u_0$ and for this we need the Weissler Theorem. If $\mathcal{Z}$ is simply $L^{3}(\mathbb {R}^3)$, the standard fixed point argument is not valid in $\mathcal{X}$.
Weissler proposes to replace $\mathcal{X}$ by the Banach space $Y\subset \mathcal{X}$ consisting of all functions such that
\begin{equation}
\label{eq1} %Ter%
\left \{ \begin{matrix} u(\cdot,t)\in C([0,\infty);L^{3}(\mathbb{R}^{3})) 
\\ 
t^{\frac{\sigma}{2}} u(\cdot,t)\in C([0,\infty);L^{p}(\mathbb{R}^{3}))
\\   \lim _{t\to 0} {t^{\frac{\sigma}{2}}\left \| u \right \|_{p}}=0
\\   \lim _{t\to \infty} {t^{\frac{\sigma}{2}}\left \| u \right \|_{p}}=0, 
\end{matrix}\right. 
\end{equation}
where $3<p<9$ and $\sigma=1-\frac{3}{p}$.
Three distinct norms will be used. As above, natural norm is
\begin{eqnarray*}
\left \| u \right \|_{\mathcal{X}} = \underset{t>0}{\sup} \  \left \| u(\cdot,t) \right \|_{3}.
\end{eqnarray*}
The second norm is called the strong norm and is defined by
\begin{eqnarray*}
\left \| u \right \|_{*} =  \left \| u \right \|_{\mathcal{X}} + \ \underset{t>0}{\sup} \ t^{\frac{\sigma}{2}}  \left \| u(\cdot,t) \right \|_{L^{p}}.
\end{eqnarray*}
The third norm is the weak norm  that is defined by
\begin{eqnarray*}
\left \| u \right \|_{Y} = \underset{t>0}{\sup} \ t^{\frac{\sigma}{2}}  \left \| u(\cdot,t) \right \|_{L^{p}}.
\end{eqnarray*}
Then, we consider the linear space $\mathcal{H}$ of all solutions $u(x,t)$, of the linear heat equation 
$\frac{\partial u}{\partial t}= \Delta u$. Then the previous three norms are equivalent on $\mathcal{H}$.
\\ 
 The first problem is show to  the global existence of the solution $u(t,x)$ for \eqref{class}  under a condition of smallness on the initial data $u_0(x)$ in homogeneous Besov spaces $\dot{B}_{p}^{-\sigma, \infty}(\mathbb{R}^{3})$, where $3<p<9$ and $\sigma=1-3/p$.  Actually, our theorem is merely an extension of a theorem of Meyer \cite{Meyer} who did this for $p=6$, and it is based on a fixed-point theorem adapted for our case and the fact that the integral $\int_{0}^{t}(t-\tau)^{-\frac{3}{p}} \tau^{-\frac{3\sigma}{2}}$ is finite for $3<p<9$. We will prove in section \ref{proofT1} the following theorem.
\begin{theorem}
\label{T1}
Let  $3<p<9$ and  let $\left \| \cdot \right \|_{\dot{B}_{p}^{ -\sigma, \infty}}$ denote the norm in the homogeneous Besov space $\dot{B}_{p}^{-\sigma, \infty}(\mathbb{R}^{3})$, with $\sigma= 1 - \frac{3}{p}$.
There exists a positive number $\eta$ such that, if the initial condition $u_0(x)$ satisfies $u_0 (x) \in L^{3}(\mathbb{R}^{3})$  and $\left \| u_0(x) \right \|_{\dot{B}_{p}^{-\sigma, \infty}} \leq \eta $,  then  there exists  a global solution $u(x,t)\in C([0,\infty),L^{3}(\mathbb{R}^{3})) \cap Y=: \mathcal{W}$ to \eqref{class}, where $(Y. \left \| \cdot \right \|_{Y} )$ is the  Banach space such that
\[
\left \| u \right \|_{Y} = \ \underset{t>0}{\sup} \ t^{\frac{\sigma}{2}}  \left \| u(\cdot,t) \right \|_{L^{p}} < \infty
\]
\end{theorem}
%(3< p \leq \infty)%
Because of the continuous embedding $L^{3} \subset \dot{B}^{-\sigma, \infty}_{p}$ (see \eqref{inega}), small data in $L^{3}(\mathbb{R}^{3})$ give rise to a global solution. The interesting feature of Theorem \ref{T1}, however, is that  the
$\dot{B}^{-\sigma, \infty}_{p}$-norm can be small even when $L^{3}$-norm is large: this is typically the case of fast oscillating data, see \cite{Canno1, Meyer}.
\\
Is it possible to further relax the smallness condition  $\left \| u \right \|_{\dot{B}^{-\sigma,\infty}_{p}} \leq \eta$ for the global solvability of \eqref{class}? Our main result, Theorem \ref{inflation} below  provides a negative answer in this direction.
\begin{theorem}
\label{inflation}
Let  $\delta>0$. Then there exists $u_0\in \mathcal{S}(\mathbb{R}^{3})$ such that the unique Weissler's solution $u$ arising from $u_0$ and belonging to $C([0,T^{*}], L^{3}(\mathbb{R}^{3}))\cap L^{\infty}_{loc}(]0,T^{*}],L^{\infty}(\mathbb{R}^{3}))$ verifies $T^{*}< \delta$. In  addition, we can choose $u_0$ in the following way: for all $3\leq q \leq +\infty$, 
\begin{eqnarray*}
 \left \| u_{0} \right \|_{\dot{B}^{-\frac{2}{3},q}_{9}} \leq \delta.  
\end{eqnarray*}
\end{theorem}
In particular, it follows that an initial data $u_0\in \mathcal{S}(\mathbb{R}^{3})$ and 
and arbitrarily small in $\dot{B}^{-1,\infty}_{\infty}(\mathbb{R}^{3})$ can produce solutions that explode in finite time.
In the case of the incompressible Navier-Stokes equation, a related result were obtained by Bourgain-Pavlovic \cite{Bourg} and later by Yoneda \cite{Yone}.
Bourgain-Pavlovic proved that the incompressible Navier-Stokes equation is ill-posed in the Besov space $\dot{B}^{-1,\infty}_{\infty}$ showing an inflation phenomenon of the norm of the solution from an initial condition $u_0$. Yoneda has generalized this result to the case of Besov spaces $\dot{B}^{-1,\infty}_{q}$ with $q>2$.
\\
If we compare these results with our results for the equation \eqref{class}, we can see that our result  for \eqref{class} is stronger, as it shows that an arbitrarily small initial data $u_0$ can produce a blow-up in short time of the solution, while the results of Bourgain, Pavlovic and Yoneda only show that an arbitrarily small initial data $u_0$ can produce arbitrarily large solutions in short time.
\\
Our demonstration is inspired to that of  Mongomery-Smith \cite{Montg}, where he  built initial data such that there is no a reasonable solution to a toy model for the  Navier-Stokes equation in ${\dot{B}^{-1,\infty}_{\infty}}$. 
\section{Proof of Theorem \ref{T1}}
\label{proofT1}
For the proof of Theorem \ref{T1}, we use  a  fixed point argument:  
\begin{lemma}[see \cite{Meyer}]
\label{fixe}
Let $(\mathcal{W}, \left \| \cdot \right \|_{\mathcal{W}} )$  a Banach space and let
\begin{eqnarray*}
 B: \mathcal{W} \times \mathcal{W} \times \mathcal{W} \rightarrow \mathcal{W},  
\end{eqnarray*}
a application trilinear such that 
\begin{eqnarray*}
\left \| B(x,y,z) \right \|_{\mathcal{W}} \leq C_0 \left \| x \right \|_{\mathcal{W}} \left \| y \right \|_{\mathcal{W}} \left \| z \right \|_{\mathcal{W}}.
\label{3}
\end{eqnarray*}
If $\left \| x_0 \right \|_{\mathcal{W}} \leq \frac{2}{3\sqrt{3}} C_0^{-1/2}$, then the equation
\begin{eqnarray*}
x = x_0 + B(x,x,x) \ \ \ \ x\in \mathcal{W},
\label{4}
\end{eqnarray*} 
has a unique solution which satisfies  $\left \| x \right \| _{\mathcal{W}}\leq \frac{1}{\sqrt{3}} C_0^{-1/2}$ and this solution is the limit of the sequence $(x_n)_{n\in\mathbb{N}}$, defined by
\begin{eqnarray*}
x_{n+1}=x_0 + B(x_n, x_n, x_n), 
\label{5}
\end{eqnarray*} 
and therefore the function defined as 
\begin{eqnarray*}
\Psi(x_0)= \displaystyle\lim_{n \to{+}\infty}{x_n}, 
\end{eqnarray*}
is analytic in the ball  $ \left \| x \right \|_{\mathcal{W}} \leq \frac{2}{3\sqrt{3}} C_0^{-1/2}$.
\label{lema1}
\end{lemma}
\begin{proof} {\textit{Theorem \ref{T1}}}
\newline
We change \eqref{duhamel} as
\begin{eqnarray*}
u(x,t)=e^{\Delta t} u_0(x) +  \Gamma(u,u,u)(x,t)
\end{eqnarray*}
where
\begin{eqnarray}
\label{trilineal}
\Gamma(u_1,u_2,u_3) (x,t) = \displaystyle\int_{0}^{t} e^{(t-\tau) \Delta} u_1 u_2 u_3 (x,\tau) \, d\tau.
\end{eqnarray}
Moreover, using the  Young's Inequality
\begin{eqnarray*}
\left \| \Gamma(u_1,u_2,u_3)(\cdot,t) \right \|_{L^p} \leq  \displaystyle\int_{0}^{t} \left \| G_{t-\tau}(\cdot)\right \|_{L^r} \left \| u_1 u_2 u_3 (\cdot,\tau) \right \|_{L^{\frac{p}{3}}} \, d\tau,	
\end{eqnarray*}
where $r=\frac{p}{p-2}$.
\begin{eqnarray*}
 \displaystyle\int_{0}^{t}  \left \| G_{t-\tau}(\cdot)\right \|_{L^r} \left \| u_1 u_2 u_3 (\cdot,\tau) \right \|_{L^{\frac{p}{3}}} \, d\tau
&=& C \displaystyle\int_{0}^{t}   (t-\tau)^{-\frac{3}{p}}   \left \| u_1 u_2 u_3(\cdot,\tau) \right \|_{L^{\frac{p}{3}}} \, d\tau \\
&\leq& C \displaystyle\int_{0}^{t}  (t-\tau)^{-\frac{3}{p}}   \  \left \| u_1(\cdot,\tau)\right \|_{L^p}\left \| u_2(\cdot,\tau)  \right \|_{L^p}    \left \|  u_3(\cdot,\tau) \right \|_{L^p} \, d\tau  \\  
&\leq& C \displaystyle\int_{0}^{t}   (t-\tau)^{-\frac{3}{p}} \tau^{-\frac{3\sigma}{2}} \, d\tau  \  \left \| u_1\right \|_{Y}\left \|u_2  \right \|_{Y}   \left \|  u_3 \right \|_{Y}.
\label{17}
\end{eqnarray*}
As $3<p<9$, the integral $\int_{0}^{t}   (t-\tau)^{-\frac{3}{p}} \tau^{-\frac{3\sigma}{2}} \, d\tau =C t^{-\sigma/2}$. Thus, we get
\begin{eqnarray*}
\left \| \Gamma(u_1,u_2,u_3) \right \|_{Y} \leq C  \left \| u_1\right \|_{Y} \left \| u_2 \right \|_{Y} \left \| u_3 \right \|_{Y} . 
\end{eqnarray*}
Also, in \cite{Meyer} we have
\begin{eqnarray*}
\left \| \Gamma(u_1,u_2,u_3) \right \|_{3} \leq C  \left \| u_1\right \|_{Y} \left \| u_2 \right \|_{Y} \left \| u_3 \right \|_{Y},  
\end{eqnarray*} 

and $\Gamma(u_1,u_2,u_3)(\cdot,t)\in C([0,\infty);L^{3}(\mathbb{R}^{3}))$.
On the other hand, we know that  the initial data $u_0$ belongs to $\dot{B}_{p}^{-\sigma, \infty}(\mathbb{R}^{3})$, if and only if $\left \|e^{\Delta t} u_0 \right \|_{p} \leq c \ t^{-\frac{\sigma}{2}}$. Therefore, we choose $\eta=\frac{2C^{-1/2}}{3\sqrt{3}}$ and 
\begin{eqnarray}
\left \|  u_0 \right \|_{\dot{B}_{p}^{ -\sigma, \infty}}\leq \frac{2C^{-1/2}}{3\sqrt{3}}=\eta,
\end{eqnarray}
which allow us to apply $\eta$ Lemma \ref{lema1} in $\mathcal{W}$ we can conclude.
\end{proof}
\begin{remark}
\label{R1}
If we consider $u_0$ as Theorem \ref{T1} and $\widehat{u_0}(\xi)$ positive, then $u(x,t)$ solution of \eqref{class} has its positive Fourier transform.   
Indeed, suppose for a contradiction that there is $t$ such that  $\widehat{u}(\xi,t)$ changes sign. By Theorem \ref{localexis}, we know that  there is $T(u_0)>0$, such that  $u$  is the unique solution of  \eqref{class} in $C([0,T(u_0)],L^{3}(\mathbb{R}^{3})) \cap L^{\infty}_{loc}((0,T(u_0)],L^{\infty}(\mathbb{R}^{3}))$. We set
\begin{eqnarray*}
t_0 =  \underset{}{\inf}  \{ t\in [0,T(u_0)) , \exists   \ \xi  \ ; \    \widehat{u}(\xi,t) < 0 \}. 
\end{eqnarray*}
We must have $t_0 > 0$ by constructing a local solution by fixed point.
Then, $\forall \  0\leq t < t_0$ and by continuity of the  positive function $(\xi,t) \rightarrow{\widehat{u}}(\xi,t)$, we have $\widehat{u} (\xi,t_0) \geq 0$.
But the  Cauchy's problem with $\widehat{u_0} =  \widehat{u} (\xi,t_0)$ has a solution $v$  in $[t_0, t_0 +\alpha)$ ($\alpha>0$), obtainable by fixed point  as:
\begin{eqnarray*}     
\widehat{v}(\xi,t) = e^{-t \left | \xi \right |^{2}} \widehat{u} (\xi,t_0 ) +\displaystyle\int_{0}^{t}  \widehat{v} * \widehat{v} * \widehat{v} (\xi,s) \, ds,
\end{eqnarray*}
then $\widehat{v}(s,\xi) > 0$, with $s\in [t_0, t_0+\alpha)$. Also, by the uniqueness of solution in $C([0, T(u_0)), L^{3}(\mathbb{R}^{3})) \cap L^{\infty}_{loc} ([0, T(u_0)],L^{\infty}(\mathbb{R}^{3}))$, this solution coincides with $u$ dans $[t_0, t_0 +\alpha)$. Then there exists  $\alpha >0$ such that, $\widehat{u} \geq 0$ in $[0, t_0 +\alpha)$. This is absurd by the maximality of $t_0$.  
\end{remark}
\begin{remark}
In \cite{Mia}, Miao, Yuan and Zhang have generalized this result. They studied the Cauchy problem for the nonlinear heat equation \eqref{Ter} in homogeneous Besov spaces $\dot{B}^{s,p}_{r}(\mathbb{R}^{ n})$, with $s<0$. .
The non-linear estimation is established by means of trichotomy Littlewood-Paley and is used to prove the global existence of the solutions for small initial data in the homogeneous Besov space $\dot{B}^{s,p}_{r}(\mathbb{R}^{ n})$ and with $s=\frac{n}{p}-\frac{2}{b}$, with $b>0$. In particular, when $r = \infty$ and when the initial data $u_0$ satisfies  $\lambda^{\frac{2}{b}}u_0(\lambda x)=u_0(x)$ for all $\lambda>0$,, the main result in \cite{Mia} leads to the global existence of  self-similar solutions of the  problem \eqref{Ter}.
\end{remark}
%%%%falta mas del teorema punto
\section{Proof of main Theorem and some comments}
Our demonstration of the explosion in finite time for the solution of \eqref{class} is based on the construction of a suitable initial condition $u_0\in \mathcal{S}(\mathbb{R}^{3})$: the corresponding solution satisfies $u(t,\cdot) \in L^{1}(\mathbb{R})$ for $t\in [0,T^{*}]$ by Remark \ref{L1}, then, we use the Fourier transform $\widehat{u}(\xi,t)$ of the solution and under certain conditions, we show the finite blow-up of  $\widehat{u}(\cdot,t)$ in $L^{\infty} (\mathbb{R}^{3})$. This fact implies the finite time blow-up of $u(x,t)$ in $L^{1}(\mathbb{R}^{3})$.
\\
There are many blowup results based on the maximum principle but to our knowledge, our blow-up criterion of the solution of \eqref{class} is the only one that uses the positivity of the Fourier transform inherited from its initial condition $u_0$.
\\
Now, we are going to formulate a useful lemma for the construction of the initial condition $u_{0,N}$ that allows us to demonstrate the main Theorem \ref{inflation}.
\begin{lemma}
\label{lemme1}
Let  $\delta>0$ and $w$ be a tempered Schwartz  function $\mathcal{S}(\mathbb{R}^{3})$, such that $\widehat{w}(\xi) \geq 0 \ \forall \xi$ and $\widehat{w}(\xi)$ is an even function.  Also assume that the support of
$\widehat{w}$ is in $B_1 (0)$. Let $w_k = w^{3^{k}}$ and $\alpha_k(t)= 3^{\frac{3}{2}+k -\frac{3^{k+1}}{2}}  \  c_{\delta}^{\frac{1}{2}(3^{k}-1)} \  e^{-3^{k}t} \ \mathds{1}_{t\geq t_k}$, where $t_0=0$, $t_k = 4\delta  \sum_{j=1}^k 3^{-2j}$, $c_\delta= 1 -e^{-4\delta}$ and $\mathds{1}_{t\geq t_k}$ is the Indicator function of the interval $[t_k,t]$. Then, if $u$ is the solution of \eqref{duhamel} with initial condition $u_0(x) \in L^{3}(\mathbb{R}^{3})$ such that $\widehat{u}_0 (\xi, t) \geq A \widehat{w}$ with $A > 0$, then 
\[
\widehat{u} \geq A^{3^{k}} \alpha_k (t)  \ \widehat{w}_k (\xi)  \  \forall k \geq 0. 
\]
\end{lemma}
\begin{proof}
Using Fourier  transform, we have that \eqref{duhamel} becomes 
\begin{eqnarray}
\widehat{u}(\xi, t) = e^{-t  \left | \xi \right |^{2}} \widehat{u}_0(\xi) + \displaystyle\int_{0}^{t} e^{(s-t) \left | \xi \right |^{2}} \widehat{u}(s,\xi) * \widehat{u}(s,\xi) * \widehat{u}(s,\xi)    \, ds
\end{eqnarray}
We start with  the case $k=0$: $e^{-t  \left | \xi \right |^{2}} \widehat{u}_0(\xi) > 0$  because $\widehat{u}_0(\xi) \geq A \widehat{w}(\xi) > 0$. Then, as $ \widehat{u}(\xi,t)\geq0$, using that supp $\widehat{w} \ \subset \{ \left | \xi \right | \leq 1\}$, we get
\begin{eqnarray}
\label{rec}
\widehat{u}(\xi, t) \geq e^{-t  \left | \xi \right |^{2}} \widehat{u}_0(\xi) \geq e^{-t \left | \xi \right |^{2}} A \widehat{w}(\xi)\geq A   \ e^{-t }  \widehat{w}(\xi) \ \ \forall  t > 0.
\end{eqnarray}
Suppose that our desired inequality holds for $k-1$. Then we get, for all $t\geq t_k$:  
\begin{equation*}
\widehat{u}(\xi, t)    \geq \displaystyle\int_{0}^{t} e^{(s-t) \left | \xi \right |^{2}} \widehat{u}(s,\xi) * \widehat{u}(s,\xi) * \widehat{u}(s,\xi) \, ds
\end{equation*}
\begin{eqnarray*}
 \ \ \ \  \ \ \ \ \ \ \ \ \ \  \ \ \ \ \ \ \ \ \ \ \ \ \ \  \ \ \ \  &\geq& \displaystyle\int_{0}^{t} e^{(s-t) \left | \xi \right |^{2}} (A^{3^{k-1}} \alpha_{k-1}(s))^{3} \widehat{w}_{k-1} * \widehat{w}_{k-1}* \widehat{w}_{k-1}(\xi)    \, ds  \\
&=&  A^{3^{k}} \widehat{w}_k (\xi)  \displaystyle\int_{0}^{t} e^{(s-t) \left | \xi \right |^{2}} \alpha_{k-1}^{3} (s)  \,ds \\
&\geq& A^{3^{k}} \widehat{w}_k (\xi) \  3^{\frac{3}{2}+3k-\frac{3^{k+1}}{2}} c_{\delta}^{\frac{1}{2}(3^{k}-3)}   \displaystyle\int_{t_{k-1}}^{t}  e^{-3^{k}s} e^{3^{2k} (s-t)} \,ds \\
&\geq& A^{3^{k}} \widehat{w}_k (\xi) \  3^{\frac{3}{2}+3k-\frac{3^{k+1}}{2}} c_{\delta}^{\frac{1}{2}(3^{k}-3)}    e^{-3^{k}t}  \displaystyle\int_{t_{k-1}}^{t} e^{3^{2k} (s-t)} \,ds \\
&\geq& A^{3^{k}} \widehat{w}_k (\xi) \  3^{\frac{3}{2}+3k-\frac{3^{k+1}}{2}} c_{\delta}^{\frac{1}{2}(3^{k}-3)}   \  e^{-3^{k}t} \  3^{-2k}  \  c_\delta \\
&\geq& A^{3^{k}} \widehat{w}_k (\xi) \  3^{\frac{3}{2}+k-\frac{3^{k+1}}{2}} c_{\delta}^{\frac{1}{2}(3^{k}-1)}   \  e^{-3^{k}t} \  \\
&=&  A^{3^{k}} \alpha_k (t)  \ \widehat{w}_k (\xi). 
\end{eqnarray*}
because $t\geq t_{k}$, with $t_k - t_{k-1} \geq 3^{-2k} \ 4 \delta$, then $1- e^{3^{2k}(t_{k-1}-t)} \geq \ c_\delta$. Our claim now follows by induction.
\end{proof}
Next lemma provides a first blowup result for equation \eqref{duhamel}. 
%Here as in the  article of Mongomery-Smith \cite{Montg}, we prove that there is  a  initial data $u_0\in\mathcal{S}(\mathbb{R}^{3})$ (and hence in every Triebel-Lizorkin, Besov, Lebesgue or Sobolev space), such that  there exists $0<T^{*}<\infty$  where the solution $u(T^*,\cdot)$ arising from $u_0$ is  not in  Triebel-Lizorkin, or Besov space (and hence in no Lebesgue or Sobolev space), then we have the following lemma.
\begin{lemma}
\label{Theoreme1}
Let  $\delta>0$ and $w\in\mathcal{S}(\mathbb{R}^{3})$ $(w\neq 0)$ be a  Schwartz  function such that $\widehat{w}(\xi) \geq 0 \ \forall \xi$  and $\widehat{w}(\xi)$ is an even function. Also assume that the support of
$\widehat{w}$ is in $B_1 (0)$. Let $u_0\geq Aw$, with $A\geq \frac{3^{\frac{3}{2}} \ c_\delta^{-\frac{1}{2}} \ e^{\frac{\delta}{2}}}{\left \| \widehat{w} \right \|_{L^1}}$, with $c_\delta=(1-e^{-4\delta})$. If  $u$ is  the unique  Weissler's solution of \eqref{duhamel} arising from $u_0$ and belonging to $C([0,T^{*}], L^{3}(\mathbb{R}^{3}))\cap L^{\infty}_{loc}((0,T^{*}],L^{\infty}(\mathbb{R}^{3}))$, then $T^{*}\leq \frac{\delta}{2}$.
\end{lemma}
\begin{proof}
Assuming $T^{*}>\frac{\delta}{2}$ (otherwise the conclusion readily follows), applying Lemma \ref{lemme1}, and using that  $t_k \uparrow \frac{\delta}{2}$ as $k\rightarrow+\infty$, we get:
\begin{eqnarray*}
 \underset{0\leq t \leq \frac{\delta}{2}}{\sup} \left \| u(t,\cdot) \right \|_{L^{1}} &=& \underset{0\leq t \leq \frac{\delta}{2}}{\sup} \left \| \widehat{u}(t,\cdot) \right \|_{L^{\infty}} \\
 &\leq&  \underset{k\in \mathbb{N}}{\sup} \ A^{3^{k}} 3^{3/2+ k - \frac{3^{k+1}}{2}} e^{-3^{k} \delta/2} \left \| \widehat{w} \right \|_{L^{1}}^{3^{k}}.
\end{eqnarray*}
In the first equality we used the positivity of $\widehat{u}(t,\cdot)$. It is clear that the right-hand side is infinite if $A\geq\frac{3^{3/2} c_{\delta}^{-1/2} e^{\delta/2}}{\left \| \widehat{w} \right \|_{L^{1}}}$. The conclusion then follows by Remark \ref{R1}. 
\end{proof}
\begin{remark}
The blowup result of Lemma \ref{Theoreme1} does not immediately imply Theorem \ref{inflation}, because in this Lemma the condition $\left \| u_{0} \right \|_{\dot{B}^{-\frac{2}{3},q}_{9}} \leq \delta$ is not satisfied.
\end{remark}
In the last part, we will prove the main theorem of our article.  
\begin{proof}{\textit{Theorem \ref{inflation}}}
\newline
Let $\delta>0$ fixed and $w \in\mathcal{S}(\mathbb{R}^{3})$ such that $\widehat{w}\neq0$ and $\widehat{w}(\xi) \geq 0 \ \forall \xi$.  Also assume that $\widehat{w}$ is an even function and its  support is in $B_{\frac{1}{6}}(\frac{1}{6}e_1)$.
% Also assume that  in $B_{\frac{1}{3}}(0)$.
On the another hand, let $u_{0,N}\in \mathcal{S}(\mathbb{R}^{3})$, defined as  
\begin{eqnarray*}
u_{0,N} (x) = \epsilon_{N}  \displaystyle\sum_{k=1}^N   2^{2/3 \ k} \ \eta_{k}  \ \cos ((2^{k}-1) x_1) \ w(x)   \qquad \mbox{where} \qquad \eta_k = k^{-1/3} \ \ \mbox{and} \ \  \epsilon_{N} = 1/ \log(\log(N)),
\end{eqnarray*}
with $N\in \mathbb{N}$. Then, by the  Theorem \ref{localexis}, there is  $T^{*}_{N}>0$ and  a unique solution $u_{N}(x,t)$ to \eqref{class} arising from  $u_{0,N}(x)$ such that   
 $u_{N} \in C([0, T^{*}_{N}), L^{3}(\mathbb{R}^{3})) \cap L^{\infty}_{loc} (]0, T^{*}_{N}],L^{\infty}(\mathbb{R}^{3}))$. Moreover, we can see that $(\eta_k)_{k\in \mathbb{N}} \not\in\ell^{3}$, but $(\eta_k)_{k\in \mathbb{N}} \in\ell^{q}$, with $q>3$ and that $\epsilon_{N}$ slowly converges to $0$.
We do the Littlewood-Paley analysis observing  that $\Delta_j(\cos(2^{k}-1x_1) w(x))=0$ for all $j\in\mathbb{Z}$ and $k=1,..,N$, except when $j$ and $k$ are of the same order. Then, we get 
\begin{eqnarray*}
\Delta_j  u_{0} = \left \{ \begin{matrix}     \epsilon_{N}  \  2^{2/3 \ j}  \ \eta_j \  w(x) \cos((2^{j}-1)x_1)  &  j = 0,1,...,N   \ \  
\\ 
\\
0  & \mbox{otherwise} \end{matrix}\right.   
\end{eqnarray*}
Thus, if $q>3$ we get 
\begin{eqnarray*}
\left \| u_{0,N} \right \|_{\dot{B}^{-\frac{2}{3},q}_{9}}&\simeq&     \epsilon_{N} \left( \displaystyle\sum_{j\in\mathbb{Z}}^{} 2^{-\frac{2}{3}q j} \left \| \Delta_j \ u_{0,N} \right \|^{q}_{9}  \right)^{\frac{1}{q}}\simeq   \epsilon_{N} \left( \displaystyle\displaystyle\sum_{j=1}^{N} 2^{-\frac{2}{3}q j +\frac{2}{3}q j}  \eta_j^{q}\left \| w(x) \cos ((2^{j}-1) x_1) \right \|^{q}_{9} \right)^{\frac{1}{q}}\\
&\leq&     \epsilon_{N}  \left( \displaystyle\displaystyle\sum_{j=1}^{N}  \eta_j^{q} \left \| w(x) \right \|^{q}_{9} \right)^{\frac{1}{q}} =   \epsilon_{N}   \left \| w(x) \right \|_{9} \left( \displaystyle \displaystyle\sum_{j=1}^N \eta_j^{q}	 \right)^{\frac{1}{q}} \longrightarrow{} 0 \qquad \mbox{when}  \ \   N \rightarrow{} +\infty.
\end{eqnarray*}
Thus, for all $q>3$ fixed,  there exists  $N^{*} \in \mathbb{N}$ such that 
\begin{eqnarray*}
\left \| u_{0,N^{*}} \right \|_{\dot{B}^{-\frac{2}{3},q}_{9}} \leq \delta.
 \end{eqnarray*}
 If the lifetime $T_{N^{*}}^{*}$ of the solution of \eqref{class} arising from $u_{0,N^{*}}$ is less than $\delta$, then there is nothing to prove. Therefore, we can assume $T^{*}_{N^{*}}\geq \delta$. 
 To simplify the notation, from now on we set we call $N^*=N$. By Remark \ref{R1}, we have  $\widehat{u}_{N}(t,\xi)  \geq 0$  $ \forall t \in [0, T_{N}^{*}]$. Thus, if $ 0<t \leq T_{N}^{*}$, we get
\begin{eqnarray*}
\widehat{u}_{N} (t,\xi) &=& e^{-t \left | \xi \right |^{2} } \widehat{u}_{0,N} (\xi)+\displaystyle\int_{0}^{t}  e^{-(t-s)\left | \xi \right |^{2}}  [\widehat{u}_N * \widehat{u}_N * \widehat{u}_N (\cdot,s)](\xi) \, ds \geq e^{-t \left | \xi \right |^{2} } \widehat{u}_{0,N} (\xi) \\ 
&=& \epsilon_{N}    \left( \displaystyle \displaystyle\sum_{k=1}^N 2^{\frac{2}{3}k} \ \eta_k \  e^{-t \left | \xi \right |^{2} } \frac{1}{2} (\widehat{w}(\xi + (2^{k}-1) e_1)+\widehat{w}(\xi- (2^{k}-1) e_1)) \right) \\
&\geq& \epsilon_{N}      \ \left( \displaystyle \displaystyle\sum_{k=1}^N 2^{\frac{2}{3}k - 1}  \ \eta_k  \  e^{-t \ 2^{2k} } (\widehat{w}(\xi + (2^{k}-1) e_1)+\widehat{w}(\xi- (2^{k}-1) e_1)) \right).
\end{eqnarray*}
We have 
\begin{eqnarray*}
[\widehat{u}_{N} * \widehat{u}_{N} * \widehat{u}_{N} (\cdot,s)](\xi) &\geq&   \epsilon_{N} \left( \displaystyle \displaystyle\sum_{k=0}^{N-1} 2^{\frac{2}{3}k - \frac{1}{3}}  \ \eta_{k+1}   \  e^{-s 2^{2k+2} } \widehat{w}(\xi+ (2^{k+1}-1) e_1) \right)  \\
 &*&    \epsilon_{N} \left( \displaystyle \displaystyle\sum_{k=1}^N 2^{\frac{2}{3}k - 1}  \  \eta_k \   e^{-s 2^{2k} } \widehat{w} (\xi- (2^{k}-1) e_1)\right) \\  
&*&   \epsilon_{N} \left( \displaystyle \displaystyle\sum_{k=1}^N 2^{\frac{2}{3}k - 1}  \eta_k \   e^{-s 2^{2k} } \widehat{w} (\xi- (2^{k}-1) e_1) \right) \\
&\geq&  \epsilon_{N}^{3}  \left( \displaystyle \displaystyle\sum_{k=1}^{N-1}  2^{\frac{2}{3}k - \frac{1}{3} +2(\frac{2k}{3} -1)} e^{-s(2^{2k+2} +2^{2k}+ 2^{2k} )}   \ \eta_{k}^{2}  \  \eta_{k+1} \right) \widehat{w} * \widehat{w} * \widehat{w} (\xi+e_1). 
\end{eqnarray*}
Observe that $\widehat{w} *\widehat{w} * \widehat{w}(\cdot+e_1)$ is supported by $B_1(0)$. We have 
\begin{eqnarray*}
\widehat{u}_{N} (t,\xi)  &\geq&   \displaystyle\int_{0}^{t} e^{-(t-s) \left | \xi \right |^{2}}  [\widehat{u} * \widehat{u} * \widehat{u} (\cdot,s)](\xi) \, ds \\
&\geq& \epsilon_{N}^{3} \displaystyle\int_{0}^{t} e^{-(t-s)}    \left( \displaystyle \displaystyle\sum_{k=1}^{N-1} 
 2^{2k-\frac{7}{3}}   e^{-s (3.2^{2k+1})}   \eta_{k}^{2}  \  \eta_{k+1} \right)  \ ( \widehat{w} * \widehat{w} * \widehat{w}) (\xi+e_1) \, ds \\
 &=& \left( \displaystyle\epsilon_{N}^{3} \displaystyle\sum_{k=1}^{N-1} 
 2^{2k-\frac{7}{3}}  \eta_{k}^{2}  \  \eta_{k+1} e^{-t}   \displaystyle\int_{0}^{t} e^{(1- 3.2^{2k+1})s}  \, ds  \right) \ ( \widehat{w} * \widehat{w} * \widehat{w}) (\xi +e_1) \\
 &=& \left( \displaystyle\epsilon_{N}^{3} \displaystyle\sum_{k=1}^{N-1} 
 \frac{2^{2k-\frac{7}{3}}}{ 3.2^{2k+1}-1} \   \eta_{k}^{2}  \  \eta_{k+1} \  e^{-t}   (1- e^{t(1- 3.2^{2k+1})} \ )  \right) \ ( \widehat{w} * \widehat{w} * \widehat{w}) (\xi+e_1).
\end{eqnarray*}
Choose $t=\frac{\delta}{2}$. Therefore, if we call $\tau_{N}>0$ as
 \begin{eqnarray*}
\tau_{N} = \epsilon_{N}^{3} \displaystyle\sum_{k=1}^{N-1} \frac{2^{2k-\frac{7}{3}}}{ 3.2^{2k+1}-1} \   \eta_{k}^{2}  \  \eta_{k+1} e^{-2\delta}   (1- e^{\frac{\delta}{2}(1- 3.2^{2k+1})}), 
 \end{eqnarray*}
%and if we take $t=\frac{\delta}{2}$,
 we get
\begin{eqnarray}
\label{ler}
 \left \| \widehat{u}_{N}\left(\displaystyle\frac{\delta}{2},\cdot\right) \right \|_{L^{1}}  \geq  \tau_{N} \left \| \widehat{w} \right \|_{L^{1}}^{3}
 \end{eqnarray}
On the other hand, we consider the Cauchy problem \eqref{class} with initial data $u_{N}\left(\displaystyle\frac{\delta}{2},x\right)$. 
Now we apply Lemma \ref{Theoreme1} with the new initial data $u_{N}\left(\displaystyle\frac{\delta}{2},x\right)$ instead of $u_0$. Moreover  as $T_N \rightarrow + \infty$ as $N \rightarrow + \infty$, if $N\in \mathbb{N}$ is chosen large enough, then all the assumption of this Lemma are satisfied  and therefore $T^{*}_{N} \leq \delta$.  
\end{proof}
%\begin{remark}
%As $\tau_{N} \sim  C e^{-2\delta} \frac{\log N}{(\log(\log N))^{3}}$, a sufficient condition for $N$ verifies the Lemma \ref{Theoreme1} is
%\begin{eqnarray*}
%N \geq \left[\exp\left(C \frac{ e^{\frac{5\delta}{2(1-\alpha)}}}{(1-e^{-4\delta})^{\frac{1}{2(1-\alpha)}}}\right)\right], 
%\end{eqnarray*}
%where, $[\cdot]$ is the  floor function  and $\alpha\in(0,1)$.
%\end{remark}
\begin{remark}
With our choice of initial data $u_{0,N}$ we have
\begin{eqnarray*}
\left \| u_{0,N} \right \|_{\dot{B}^{-\frac{2}{3},3}_{9}} \leq   \epsilon_{N}   \left \| w(x) \right \|_{9} \left( \displaystyle \displaystyle\sum_{j=1}^N \eta_j^{3} \right)^{\frac{1}{3}} \longrightarrow{} +\infty.
\end{eqnarray*}
and 
\[
\displaystyle\sum_{k=1}^{N} \eta_k^{3}  \sim \log N  \qquad \mbox{when}  \qquad    N \rightarrow{} + \infty.  
\] 
Therefore, our initial condition in Theorem \ref{inflation} has the defect that is not in arbitrarily small in $\dot{B}^{-\frac{2}{3},3}
_{9}$. We leave open the following question: a smallness condition on $u_0$ in $\dot{B}^{-\frac{2}{3},3}_{9}$ does it imply the existence of a global solution or not ?. However Theorem \ref{inflation} shows the optimality of the assumption $p<9$ in the global existence Theorem \ref{T1}.
\end{remark} 
On a related subject, Pierre-Gilles Lemari\'e Rieusset dans \cite{Lema} studied the parabolic semi-linear equations on $(0, \infty) \times \mathbb{R}^{n}$ of type
\begin{equation*}
\partial_t u - (-\Delta)^{\frac{\alpha}{2}} u= (-\Delta)^{\frac{\beta}{2}} u^{2},
\end{equation*}
where $0 <\alpha< n +2\beta$ and $0 <\beta < \alpha$. Actually Lemari\'e-Rieusset worked on a more general quadratic nonlinearity that $(-\Delta)^{\frac{\beta}{2}} u^{2}$. He showed similar results to ours when $\beta<\frac{\alpha}{2}$.
\section*{Acknowledgement}
This work is supported by  the \textit{Secretar\'ia Nacional de Educaci\'on Superior, Ciencia, Tecnolog\'ia e Innovacio\'n del Ecuador (SENESCYT).}


\begin{bibdiv}
\begin{biblist}

\end{biblist}
\end{bibdiv}


\begin{thebibliography}{}
\bib{Ball}{article}{  
title={Remarks on blow-up and nonexistence theorems for nonlinear evolution equations.},
author={Ball, JM},
journal={The Quarterly Journal of Mathematics},
volume={28},
number={4},
publisher={Oxford University Press},
pages={473--486},
year={1977},
} 
\bib{Ban}{article}{  %3
title={Fujita type results for convective-like reaction diffusion equations in exterior domains},
author={Bandle C.},
author={Levine H.},
journal={ ZAMP},
volume={40},
number={5},
publisher={Birkhäuser-Verlag},
pages={665--676},
year={1989}
}
\bib{Ban1}{article}{  %4
title={On the existence and nonexistence of global solutions of reaction-diffusion equations in sectorial domains},
author={Bandle C.},
author={Levine H.}
journal={Trans. Am. Math. Soc.},
  volume={316},
  number={2},
  pages={595--622},
  year={1989}
}
%\bib{Beb}{article}{
  % author={J. Bebernes},
  % author={ D. Eberly},
   %title={Mathematical problems from Combustion theory.},
   %journal={Applied Math. Sci. Series}, 
   %volume={83},
  %date={1989},
%}
%\bib{Bet}{article}{
  % author={M. Betterton},
   %author={ M. Brenner},
   %title={Collapsing bacterial cylinders.},
   %journal={Physical Review E}, 
   %publisher={APS},
   %number={6},
  %pages={061904},
  %volume={64},
  %date={2001},
%}
\bib{Bourg}{article}{
   author={J. Bourgain},
   author={N. Pavlovi{\'c}},
   title={Ill-posedness of the Navier--Stokes equations in a critical space in 3D},
   journal={Journal of Functional Analysis}, 
   publisher={Elsevier},
   volume={255},
    number={9},
   date={2008},
   pages={2233--2247}}
   \bib{BreCaz}{article}{
  title={A nonlinear heat equation with singular initial data},
  author={Brezis H.},
  author={Cazenave T.},
  journal={Journal D'Analyse Mathematique},
  volume={68},
  number={1},
  pages={277--304},
  year={1996},
  publisher={Springer}
}
%\bib{Corr}{article}{
  %title={Global solutions of some chemotaxis and angiogenesis systems in high space dimensions},
  %author={Corrias, L.},
  %author={ Perthame, B.},
  %author={  Zaag, H.},
  %journal={Milan Journal of Mathematics},
  %volume={72},
  %number={1},
  %pages={1--28},
  %year={2004},
  %publisher={Springer}
%}
\bib{Canno1}{article}{ 
booktitle={Wavelets, paraproducts and Navier-Stokes With a preface by Yves Meyer.},
series={Diderot Editeur, Paris},
title={Harmonic analysis tools for solving the incompressible Navier-Stokes equations.},
author={Cannone. M.},
%pages={161-244},
year={1995},
}   
\bib{Canno}{article}{ 
booktitle={Handbook of mathematical fluid dynamics. Vol. III},
series={North-Holland, Amsterdam},
title={Harmonic analysis tools for solving the incompressible Navier-Stokes equations.},
author={Cannone. M.},
pages={161-244},
year={2004},
}

\bib{Ebd}{article}{
  title={Construction and stability of a blow up solution for a nonlinear heat equation with a gradient term},
  author={Ebde, M. A.},
  author={  Zaag, H.},
  journal={SeMA Journal},
  volume={55},
  number={1},
  pages={5--21},
  year={2011},
  publisher={Springer}
  }
%  \bib{Fish}{article}{
  %title={The wave of advance of advantageous genes},
  %author={Fisher. R.A},
  %journal={ Annals of Eugenics},
  %number={7},
  %pages={355--369},
  %year={1937},
  %}
  \bib{Frie}{article}{  
title={Remarks on nonlinear parabolic equations.},
author={Friedman. A},
journal={ Proc. Sympos. Appl. Math},
volume={Vol. XVII},
number={17},
publisher={Amer. Math. Soc.},
pages={3--23},
year={1965}
}
\bib{Frie1}{article}{  
title={Partial differential equations},
author={Friedman. A},
journal={Robert E. Krieger Publishing Co},
volume={Vol. XVII},
publisher={ Huntington,},
pages={3--23},
year={1976}
}
\bib{Fuji}{article}{
  title={On the blowing up of solutions of the Cauchy problem for $u_t=\Delta u + u^{1+\alpha}$},
  author={Fujita. H},
  journal={Journal of the Faculty of Science of the University of Tokyo},
  number={13},
  pages={109--124},
  year={1966},
  }
  \bib{Fuji1}{article}{
  title={On some nonexistence and nonuniqueness theorems for nonlinear parabolic equations},
  author={Fujita H.},
  journal={Proc. Symp. Pure Math. XVII Am. Math. Soc.},
  number={18},
  pages={105--113},
  year={1970},
  }
  
  \bib{Fur}{article}{
  title={Unicit{\'e} dans $L^3(\mathbb{R}^{3})$ et d'autres espaces fonctionnels limites pour Navier-Stokes.},
  author={ Furioli. G},
  author={Lemari{\'e}-Rieusset. P. G.},
  author={Terraneo. E.},
  journal={Rev. Mat. Iberoamericana},
  volume={16},
  pages={605-667},
  year={2000},
}
   
\bib{Gala84}{article}{
  title={On approximate self-similar solutions of a class of quasilinear heat equations with a source},
  author={V. Galaktionov },
  author={ S. Kurdyumov },
  author={A.  Samarski{\u\i}},
  journal={Mathematics of the USSR-Sbornik},
  volume={52},
  number={1},
  pages={163-188},
  year={1984},
  publisher={IOP Publishing}
}

%\bib{Gala92}{article}{
  %title={Regional Blow-Up in a Semilinear Heat Equation with Convergence to a Hamilton-Jacobi Equation},
  %author={V. Galaktionov },
  %author={ J. Vazquez},
  %journal={SIAM J. Math. Anal.},
  %volume={24},
  %number={5},
  %pages={1254--1276},
  %year={1993},
%}
\bib{Gala}{article}{
   author = {V. Galaktionov},
   author = {S. Kurdyumov},
   author = { P. Mikhailov},
author = {A. Samarskii},
   title={Blow-Up in Quasilinear Parabolic Equations},
   journal={SIAM Review}, 
   volume={38},
    number={4},
   date={1996},
   pages={692-694}} 

   \bib{Grois}{article}{
  title={On the dependence of the blow-up time with respect to the initial data in a semilinear parabolic problem},
  author={Groisman, P.},
  author={Rossi, J.},
  author={  Zaag, H.},
  year={2003},
  publisher={Taylor \& Francis}
}
\bib{Haraux}{article}{
  title={Non-uniqueness for a semilinear initial value problem},
  author={ Haraux A.},
  author={ Weissler F.},
  journal={Indiana University Mathematics Journal},
  volume={31},
  number={2},
  pages={167--189},
  year={1982}
}
\bib{Hay}{article}{
  title={On nonexistence of global solutions of some semilinear parabolic differential equations},
  author={Hayakawa K.},
  journal={Proceedings of the Japan Academy},
  volume={49},
  number={7},
  pages={503--505},
  year={1973}
}
\bib{Henr}{article}{  
title={Geometric theory of semilinear parabolic equations.},
author={Henry D.},
volume={840},
year={1985}
}
%\bib{Her96}{article}{
   %author={M. Herrero},
   %author={J. Velázquez},
   %title={ Singularity patterns in a chemotaxis model.},
   %journal={Math. Ann.}, 
   %number={3},
   %volume={306},
   %pages={583--623},
   %date={1996},
%}

\bib{Khe}{article}{
  title={Continuity of the blow-up profile with respect to initial data and to the blow-up point for a semilinear heat equation.},
  author={Khenissy, S.},
   author={Rebai, Y},
  author={  Zaag, H.},
  journal={Ann. Inst. H. Poincar\'e Anal. Non Lin\'eaire},
  volume={28},
  number={1},
  pages={1--26},
  year={2011},
  }
  \bib{Lema}{article}{ 
title={Sobolev multipliers, maximal functions and parabolic equations with a quadratic nonlinearity.},
author={Lemari{\'e}-Rieusset P. G.},
journal={Preprint, Univ. Evry},
year={2013},
}
 \bib{Levi}{article}{  
title={The critical Fujita number for a semilinear heat equation in exterior domains with homogeneous Neumann boundary values},
author={Levine,H. A.}
author={ Zhang,Q. S.},
journal={Proceedings of the Royal Society of Edinburgh: Section A Mathematics},
volume={130},
number={12},
pages={591--602},
year={2000},
}
%\bib{McKh}{article}{
   %author={H. McKean},
   %title={Application of brownian motion to the equation of kolmogorov-petrovskii-piskunov},
   %journal={Communications on Pure and Applied Mathematics}, 
   %publisher={Wiley Subscription Services, Inc., A Wiley Company},
   %volume={  28},
    %number={3},
   %date={1975},
   %pages={ 323--331}}
\bib{Meyer}{article}{
   author={Y. Meyer},
   booktitle={Mathematical Foundation of Turbulent Viscous Flows},
   title={Oscillating Patterns in Some Nonlinear Evolution Equations},
   publisher={Springer Berlin Heidelberg},
   volume={1871},
   date={2006},
   pages={101-187}}
   \bib{Mia}{article}{
   author={Miao C.},
   author={Yuan B.},
   author={Zhang B.},
   title={Strong solutions to the nonlinear heat equation in homogeneous Besov spaces.},
   journal={Nonlinear Analysis: Theory, Methods Applications}, 
   volume={ 67},
    number={5},
   date={2007},
   pages={1329--1343}
   }
   \bib{Montg}{article}{
   author={Montgomery-Smith, Stephen},
   title={Finite time blow up for a Navier-Stokes like equation},
   journal={Proceedings of the American Mathematical Society}, 
   volume={129},
    number={10},
   date={2001},
   pages={3025--3029}}
  % \bib{Mur}{article}{
   %author={J. Murray},
   %title={Mathematical biology},
   %journal={Interdisciplinary Applied Mathematics}, 
   %publisher={Springer-Verlag},
  % volume={17},
    %number={third ed.},
   %date={2002},
   %}
    %\bib{Naga}{article}{
   %author={M. Nagasawa},
   %author={N. Pavlovi{\'c}},
   %title={A limit theorem of a pulse-like wave form for a Markov process},
   %journal={Proceedings of the Japan Academy}, 
   %publisher={The Japan Academy},
   %volume={44},
    %number={6},
  %pages={491--494}
   %date={1968},
   %}
   %\bib{NiSa}{article}{
  % author={Wei-Ming Ni},
  % author={Paul Sacks},
   %title={Finite time blow up for a Navier-Stokes like equation},
   %journal={Trans. Amer. Math. Soc.}, 
   %volume={  287},
    %number={2},
   %date={1985},
   %pages={ 657-671}}
%\bib{Pao}{article}{
   %author={C.V. Pao},
   %title={Nonlinear parabolic and elliptic equations},
   %journal={Plenum Press}, 
   %publisher={Springer Science \& Business Media},
   %date={1992},
   %pages={2233--2247}}
   %\bib{Paz}{article}{  
%title={Semigroups of linear operators and applications to partial differential equations},
%author={Pazy A.},
%publisher={Springer},
%year={1983}
%}
\bib{Rau1}{article}{ %41
  title={The Fujita phenomenon in exterior domains under the Robin boundary conditions},
  author={Rault J-F.},
  journal={Comptes Rendus Math{\'e}matique},
  volume={349},
  number={19},
  pages={1059--1061},
  year={2011},
  publisher={Elsevier}
}
   \bib{Sou}{article}{
  title={Uniform blow-up profiles and boundary behavior for diffusion equations with nonlocal nonlinear source},
  author={Souplet, P.},
  journal={Journal of Differential Equations},
  volume={153},
  number={2},
  pages={374--406},
  year={1999},
  publisher={Elsevier}
}
\bib{Suz}{article}{  
title={Critical Blow-Up for Quasilinear Parabolic Equations in Exterior Domains},
author={Suzuki R.},
journal={Tokyo Journal of Mathematics},
volume={12},
number={2},
publisher={Publication Committee for the Tokyo Journal of Mathematics},
pages={397--409},
year={1996},
}
   \bib{Ter}{article}{
    author={Terraneo, Elide},
   title={Non-uniqueness for a critical non-linear heat equation},
   journal={Proceedings of the American Mathematical Society}, 
   date={2002},
  publisher={Taylor \& Francis}}
   \bib{Vaz}{article}{
  title={The Hardy inequality and the asymptotic behaviour of the heat equation with an inverse-square potential},
  author={Vazquez, J. L.}, 
  author={Zuazua, E.}, 
  journal={Journal of Functional Analysis},
  volume={173},
  number={1},
  pages={103--153},
  year={2000},
  publisher={Elsevier}
}
\bib{Vig}{article}{
  title={Spatial decay of the velocity field of an incompressible viscous fluid in $\mathbb{R}^{d}$},
  author={Vigneron, F.},
  journal={Nonlinear Analysis: Theory, Methods and Applications},
  volume={63},
  number={4},
  pages={525--549},
  year={2005},
}
\bib{Way}{article}{ 
booktitle={Hamiltonian Dynamical Systems and Applications},
series={NATO Science for Peace and Security Series},
title={Infinite dimensional dynamical systems and the Navier-Stokes equation},
author={Wayne C.},
pages={103-141},
year={2008},
}
  \bib{Weis1}{article}{
  title={Local existence and nonexistence for semilinear parabolic
equations in $L^{p}$},
  author={Weissler, Fred B},
  journal={Indiana Univ. Math. J.},
  volume={29},
  pages={79-102},
  year={1980}
  }
  \bib{Weis}{article}{
  title={Existence and non-existence of global solutions for a semilinear heat equation},
  author={Weissler, Fred B},
  journal={Israel Journal of Mathematics},
  volume={38},
  number={1-2},
  pages={29--40},
  year={1981},
  publisher={Springer}
}
\bib{Yone}{article}{  
title={Ill-posedness of the 3D-Navier-Stokes equations in a generalized Besov space near.},
author={Yoneda T.},
journal={Journal of Functional Analysis},
volume={258},
number={10},
publisher={Elsevier},
pages={3376--3387},
year={2010},
}
\end{thebibliography}
\end{document}